\documentclass{amsart}

\usepackage{graphics}
\usepackage{epsfig}
\usepackage{amsmath}
\usepackage{amssymb,fancyhdr,txfonts,pxfonts}

\usepackage{paralist}

\usepackage{pstricks,pst-text,pst-grad,pst-node,pst-3dplot,pstricks-add,pst-poly,pst-coil} 
\usepackage{pst-fun,pst-blur} 

\usepackage{hyperref}

\newtheorem{theorem}{Theorem}
\theoremstyle{plain}

\newtheorem{corollary}{Corollary}

\newtheorem{definition}{Definition}
\newtheorem{example}{Example}

\newtheorem{lemma}{Lemma}

\newtheorem{proposition}{Proposition}

\begin{document}

\title[Ordered Line\ \&\ Skew Field in Desargues Affine Plane]{Ordered Line
and Skew-Fields\\ in the Desargues Affine Plane}

\author[Orgest ZAKA]{Orgest ZAKA}
\address{Orgest ZAKA: Department of Mathematics, Faculty of Technical
Science, University of Vlora 'Ismail QEMALI', Vlora, Albania}
\email{orgest.zaka@univlora.edu.al, gertizaka@yahoo.com}

\author[James F. Peters]{James F. Peters}
\address{James F. PETERS: Computational Intelligence Laboratory, University
of Manitoba, WPG, MB, R3T 5V6, Canada and Department of Mathematics, Faculty
of Arts and Sciences, Ad\.{i}yaman University, 02040 Ad\.{i}yaman, Turkey}
\thanks{The research has been supported by the Natural Sciences \&
Engineering Research Council of Canada (NSERC) discovery grant 185986, Instituto Nazionale di Alta Matematica (INdAM) Francesco Severi, Gruppo Nazionale  per le Strutture Algebriche, Geometriche e Loro Applicazioni grant 9 920160 000362, n.prot U 2016/000036 and Scientific and Technological Research Council of Turkey (T\"{U}B\.{I}TAK) Scientific Human
Resources Development (BIDEB) under grant no: 2221-1059B211301223.}
\email{James.Peters3@umanitoba.ca}

\dedicatory{Dedicated to Girard Desargues and Emil Artin}


\begin{abstract}
This paper introduces ordered skew fields that result from the construction of a skew field over an ordered line in a Desargues affine plane. A special case of a finite ordered skew field in the construction of a skew field over an ordered line in a Desargues affine plane in Euclidean space, is also considered. Two main results are given in this paper: (1) every skew field constructed over a skew field over an ordered line in a Desargues affine plane is an ordered skew field and (2) every finite skew field constructed over a skew field over an ordered line in a Desargues affine plane in $\mathbb{R}^2$ is a finite ordered skew field.
\end{abstract}


\maketitle

\section{Introduction}

The foundations for the study of the connections between axiomatic geometry and algebraic structures were set forth by D. Hilbert \cite{Hilbert1959geometry}, recently elaborated and extended in terms of the algebra of affine planes in, for example, \cite{Kryftis2015thesis}, [3, §IX.3, p.574].
E.Artin in \cite {11} shows that any ordering of a plane geometry is equivalent to a weak ordering of its skew field. He shows that that any ordering of a Desargues plane with more than four points is (canonically) equivalent to an ordering of its field.   In his paper on ordered geometries \cite{10}, P. Scherk considers the equivalence of an ordering of a Desarguesian affine plane with an ordering of its coordinatizing division ring.    Considerable work on ordered plane geometries has been done (see, {\em e.g.}, J. Lipman~\cite{7}, V.H. Keiser~\cite{Keiser1966finiteAffinePlane}, H. Tecklenburg~\cite{Tecklenburg1991EuclideanSpaceIsAPappianSpace} and  L. A. Thomas ~\cite{Thomas1977MscDesarguesAffineHjelmslevPlane,9}).

In this paper, we utilize a method that is naive and direct, without requiring the concept of coordinates. Our results are straightforward and constructive. For this reason we begin by giving a suitable definition for our search for lines in Desargues affine planes, based on the meaning of betweenness given by Hilbert~\cite[\S 3]{Hilbert1959geometry}.



In addition, we introduce ordered skew files that result from the construction of a skew field over an ordered line in a Desargues affine plane. A special case of a finite ordered skew field in the construction of a skew field over an ordered line in a Desargues affine plane in Euclidean space, is also considered. 

Based on the works of E.Artin \cite{11} and J. Lipman \cite{7} on ordered skew fields, we prove that the  skew field that the constructed  over an ordered line in a Desargues affine plane is a ordered skew field.
To prove this, we must also give the definition of ordered Desargues affine plane, based on the definition given by E.Artin \cite {11}, but in this case the ordered line is given a suitable definition without the use of \textit{coordinates}.

Two main results are given in this paper, namely, every skew field constructed over a skew field over an ordered line in a Desargues affine plane is an ordered skew field ~\ref{thm:OrderedSkewFieldDesargues}. and every finite skew field  constructed over a skew field over an ordered line in a Desargues affine plane in R2 is a finite ordered skew field ~\ref{cor:finiteOrderedSkewFieldR2}.


\section{Preliminaries}
Let $\mathcal{P}$ be a nonempty space, $\mathcal{L}$ a nonempty subset of $\mathcal{P}$.  The elements $p$ of $\mathcal{P}$ are points and an element $\ell$ of $\mathcal{L}$ is a line.   Collinear points on a line  $\mathcal{L}$ are denoted by $\left[A,E,B\right]$, where $E$ is between $A$ and $B$.   Given \emph{distinct} points $A,B$, there is a unique line $\ell^{AB}$ such that $A,B$ lie on $\ell^{AB}$ and we write $\ell^{AB} = A +B$~\cite[p. 52]{1}.   An {\bf affine space} is a vector space with the origin removed~\cite[\S 4.1, p. 391]{BaileyCameronConnelly2008affineSpace}.  The geometric structure $\left(\mathcal{P},\mathcal{L}\right)$ is an \emph{\bf affine plane}, a subspace of an affine space, provided\\

\begin{compactenum}[1$^o$]
\item For each $\left\{P,Q\right\}\in \mathcal{P}$, there is exactly one $\ell\in \mathcal{L}$ such that $\left\{P,Q\right\}\in \ell$.

\item For each $P\in \mathcal{P}, \ell\in \mathcal{L}, P \not\in \ell$, there is exactly one $\ell'\in \mathcal{L}$ such that
$P\in \ell'$ and $\ell\cap \ell' = \emptyset$\ (Playfair Parallel Axiom~\cite{Pickert1973PlayfairAxiom}).   Put another way,
if $P\not\in \ell$, then there is a unique line $\ell'$ on $P$ missing $\ell$~\cite{Prazmowska2004DemoMathDesparguesAxiom}.

\item There is a 3-subset $\left\{P,Q,R\right\}\in \mathcal{P}$, which is not a subset of any $\ell$ in the plane.   Put another way,
there exist three non-collinear points $\mathcal{P}$~\cite{Prazmowska2004DemoMathDesparguesAxiom}.
\end{compactenum}

An affine plane is a projective plane in which one line has been distinguished~\cite{Keiser1966finiteAffinePlane}.    For simplicity, our affine geometry is on the Euclidean plane $\mathbb{R}^2$ and incident lines $\ell, \ell'$ are represented by $\ell \cap \ell'\neq \emptyset$ (intersection).   A {\bf 0-plane} is a point, a {\bf 1-plane} a line containing a minimum of 2 collinear points and a {\bf 2-plane} is an affine plane containing a minimum of 4 points, no 3 of which are collinear.   An affine geometry is a geometry defined over vector spaces $V$, field $\mathbb{F}$ (vectors are points and subspaces of $V$) and subsets $\mathcal{P}$ (points), $\mathcal{L}$ (lines) and $\Pi$ (planes)~\cite[\S 7.5]{Dillon2018affineGeometry}.   

%

\emph{\bf Desargues' Axiom, circa 1630}~\cite[\S 3.9, pp. 60-61] {Kryftis2015thesis}~\cite{Szmielew1981DesarguesAxiom}.   Let $A,B,C,A',B',C'\in \mathcal{P}$ and let pairwise distinct lines  $\ell_k , \ell_l, \ell_m, \ell^{AC}, \ell^{A'C'}\in \mathcal{L}$ such that
\begin{align*}
\ell_k \parallel \ell_l \parallel \ell_m &\ \mbox{and}\  \ell^{A}\parallel \ell^{A'}\ \mbox{and}\ \ell^{C}\parallel \ell^{C'}.\\
A,B\in \ell^{AB}, A'B'\in \ell^{A'B'}, &\ \mbox{and}\ B,C\in \ell^{BC}, B'C'\in \ell^{B'C'}.\\
A\neq C, A'\neq C', &\ \mbox{and}\ \ell^{AB}\neq \ell_{l}, \ell^{BC}\neq \ell_{l}.
\end{align*}

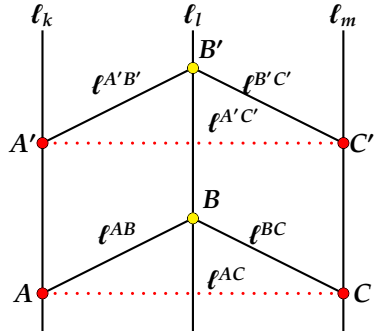
\begin{figure}[!ht]
\centering
\begin{pspicture}
(-0.5,-1.8)(4,3)
\psline[linestyle=solid](0,-1.5)(0,2.5)\psline[linestyle=solid](2,-1.5)(2,2.5)
\psline[linestyle=solid](4,-1.5)(4,2.5)
\psline[linestyle=solid](0,-1)(2,0)(4,-1) 
\psline[linestyle=dotted, , linewidth=1.2pt,linecolor = red](0,1)(4,1) 
\psline[linestyle=solid](0,1)(2,2)(4,1) 
\psline[linestyle=dotted, , linewidth=1.2pt,linecolor = red](0,-1)(4,-1) 
\psdots[dotstyle=o, linewidth=1.2pt,linecolor = black, fillcolor = yellow]%
(2,2)(2,0)
\psdots[dotstyle=o, linewidth=1.2pt,linecolor = black, fillcolor = red]%
(0,1)(0,-1)(4,1)(4,-1)
\rput(0,2.7){$\boldsymbol{\ell_k}$}\rput(2,2.7){$\boldsymbol{\ell_l}$}\rput(4.0,2.7){$\boldsymbol{\ell_m}$}
\rput(-0.25,1){$\boldsymbol{A'}$}\rput(2.25,2.25){$\boldsymbol{B'}$}\rput(4.25,1){$\boldsymbol{C'}$}
\rput(-0.25,-1){$\boldsymbol{A}$}\rput(2.25,0.25){$\boldsymbol{B}$}\rput(4.25,-1){$\boldsymbol{C}$}
\rput(1,-0.2){$\boldsymbol{\ell^{AB}}$}\rput(1,1.8){$\boldsymbol{\ell^{A'B'}}$}
\rput(3,1.8){$\boldsymbol{\ell^{B'C'}}$}\rput(3,-0.2){$\boldsymbol{\ell^{BC}}$}
\rput(2.45,-0.8){$\boldsymbol{\ell^{AC}}$}\rput(2.55,1.3){$\boldsymbol{\ell^{A'C'}}$}
\end{pspicture}
\caption[]{Desargues: $\boldsymbol{\ell^{AC}\parallel \ell^{A'C'}}$}
\label{fig:1-path}
\end{figure}

Then $\boldsymbol{\ell^{AC}\parallel \ell^{A'C'}}$.   \qquad \textcolor{blue}{$\blacksquare$}


\begin{example}
The parallel lines  $\ell^{AC}, \ell^{A'C'}\in \mathcal{L}$ in Desargues' Axiom are represented in Fig.~\ref{fig:1-path}.  In other words, the base of $\bigtriangleup ABC$ is parallel with the base of $\bigtriangleup A'B'C'$, provided the restrictions on the points and lines in Desargues' Axiom are satisfied.
\qquad \textcolor{blue}{$\blacksquare$}
\end{example}

\noindent A {\bf Desargues affine plane} is an affine plane that satisfies Desargues' Axiom. 

\begin{theorem}\label{thm:Pappus}{\bf Pappus, circa 320 B.C.}{\rm ~\cite[\S 1.4, p. 18]{Berger2010geometryRevealed}}.\\
If $[ACE] \in \ell^{EA}$, $[BFD] \in \ell^{BD}$ and $\ell^{BD},\ell^{CD},\ell^{EF}$ meet $\ell^{DE},\ell^{FA},\ell^{BC}$,
then $[NLM]$ are collinear on $\ell^{NM}$.
\end{theorem}

\begin{figure}[!ht]
\centering
\begin{pspicture}
(-0.5,-2.5)(5,4)
\psline[linestyle=solid](0,2)(5,4)\psline[linestyle=solid](0,0)(5,-2)
\rput(5.3,4.2){$\boldsymbol{\ell^{EA}}$}\rput(5.3,-2.2){$\boldsymbol{\ell^{BD}}$}
\psline[linestyle=solid](1,2.4)(4.0,-1.6)\psline[linestyle=solid](1,2.4)(2.5,-1.0)
\psline[linestyle=solid](4.0,3.6)(1,-0.4)\psline[linestyle=solid](4.0,3.6)(2.5,-1.0)
\psline[linestyle=solid](2.5,3.0)(1,-0.4)\psline[linestyle=solid](2.5,3.0)(4.0,-1.6) 
\psline[linestyle=dotted, , linewidth=1.2pt,linecolor = red](0.0,1.0)(5,1.0) 
\psdots[dotstyle=o, linewidth=1.2pt,linecolor = black, fillcolor = red]%
(1.62,1.0)(2.05,1.0)(3.15,1.0)
\rput(5.3,1.2){$\boldsymbol{\ell^{NM}\ \mbox{({\bf Pappian line}})}$} 
\psline[linestyle=solid](1,2.4)(4.0,-1.6)\psline[linestyle=solid](1,2.4)(2.5,-1.0)
\psdots[dotstyle=o, linewidth=1.2pt,linecolor = black, fillcolor = yellow]%
(1,2.4)(1,-0.4)(2.5,3.0)(2.5,-1.0)(4.0,3.6)(4.0,-1.6)
\rput(1,2.7){$\boldsymbol{E}$}\rput(2.5,3.3){$\boldsymbol{C}$}\rput(4.0,3.9){$\boldsymbol{A}$}
\rput(1,-0.7){$\boldsymbol{B}$}\rput(2.5,-1.3){$\boldsymbol{F}$}\rput(4.0,-1.9){$\boldsymbol{D}$}
\rput(1.3,1.2){$\boldsymbol{N}$}\rput(2.05,1.3){$\boldsymbol{L}$}\rput(3.5,1.2){$\boldsymbol{M}$}
\end{pspicture}
\caption[]{Pappian Line: $\boldsymbol{[NLM]\in \ell^{NM}}$}
\label{fig:Pappus}
\end{figure}
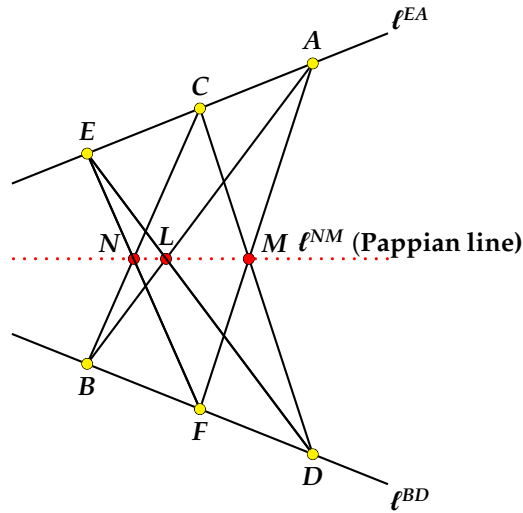

\begin{example}
The lines  $\ell^{BD},\ell^{CD},\ell^{EF}, \ell^{DE},\ell^{FA},\ell^{BC}$ in Pappus' Axiom are represented in Fig.~\ref{fig:Pappus}.  In that case, the points of intersection $[NLM]$ lie on the line $\ell^{NM}$.
\qquad \textcolor{blue}{$\blacksquare$}
\end{example}

The affine Pappus condition in Theorem~\ref{thm:Pappus} has an effective formulation relative to $[NLM]$ on line $\ell^{NM}$ given by N.D. Lane~\cite{Lane1967affinePappusCondition}, {\em i.e.},\\

{\bf Affine Pappus Condition}~\cite{Lane1967affinePappusCondition}.  Let $E,C,A,B,F,D$ be mutually distinct points as shown in Fig.~\ref{fig:Pappus} such that $A,B,C$ lie on $\ell^{EA}$ and $B,F,D$ lie on $\ell^{BD}$ and none of these points lie on $\ell^{EA}\cap \ell^{BD}, \ell^{BD}\cap \ell^{NM}$ or $\ell^{NM}\cap \ell^{EA}$.   Then
\[
\left.
\begin{array}{c}
\ell^{CB}\cap \ell^{EF}\ \mbox{lies on}\ \ell^{NM}\\ 
\ell^{AF}\cap \ell^{CD}\ \mbox{lies on}\ \ell^{NM}%
\end{array}
\right\} \Rightarrow \ell^{AB}\cap \ell^{ED}\ \mbox{lies on}\ \ell^{NM}. 
\]

\noindent This leads to the following result.

\begin{theorem}{\bf Affine Pappus Condition}{\rm ~\cite{Lane1967affinePappusCondition}}.
If the affine Pappus condition holds for all pairs of lines $\ell, \ell'$ such that $\ell \not\parallel \ell'$, then the affine Pappus condition holds for all pairs $\ell, \ell'$ with $\ell \parallel \ell'$.
\end{theorem}

Every Desarguesian affine plane is isomorphic to a coordinate plane over a field~\cite{Wedderburn1987PappusDesarguesianAffinePlanes} and every finite field is commutative~\cite[\S 3, p. 351]{Wedderburn1905DesarguesianAffinePlane}.  From this, we obtain

\begin{theorem}\label{thm:Tecklenburg}{\rm [{\bf Tecklenburg}]}{\rm~\cite{Wedderburn1987PappusDesarguesianAffinePlanes}}.\\
Every finite Desarguesian affine plane is Pappian.
\end{theorem}


   
\section{Ordered Lines and Ordered Desargues Affine Plane}

An invariant way to describe an 'order' is by means of a ternary relation:
the point $B$ \ lies "between" $A$ and $C$. Hilbert has axiomatized this
ternary relation~\cite{Hilbert1959geometry}.  In
this section, we begin by giving a suitable definition for our search for lines in
Desargues affine plane, based on the meaning of \emph{betweenness} given by D.Hilbert, {\em i.e.},
if $B$ \ lies "between" $A$ and $C$, we mark it with $[A,B,C]$. 

\begin{definition}\label{def:orderAxioms}
\bigskip An \textbf{ordered line} in a Desargues Affine plane (briefly, called 
the \emph{line}) satisfies the following axioms.
\begin{compactenum}[Lo.1]
\item For $A,B,C\in \ell ,$\textbf{\ }$\left[ A,B,C%
\right] $ $\Longrightarrow \left[ C,B,A\right] $ .

\item  For $A,B,C\in \ell $ are mutually distinct,
then we have exactly one, \ $\left[ A,B,C\right] $ , $\left[ B,C,A\right] $
or $\left[ C,A,B\right] $ .

\item For $A,B,C,D\in \ell $ , then $\left[ A,B,C%
\right] $ and $\left[ B,C,D\right] \Longrightarrow \left[ A,B,D\right] $ 
\textbf{and} $\left[ A,C,D\right] $ .

\item  For $A,B,C,D\in \ell $ , then $\left[ A,B,C%
\right] $ and $\left[ C,B,D\right] \Longrightarrow \left[ D,A,B\right] $ 
\textbf{or} $\left[ A,D,B\right] $ .
\end{compactenum}
\end{definition}

\begin{figure}[!ht]
\centering
\begin{pspicture}
(0.2,-1.)(7.,5.5)
\psline[linewidth=2.pt](1.,4.)(7.,4.)
\rput[tl](2.68,3.9){$B$}\rput[tl](4.72,3.9){$A$}\rput[tl](5.76,3.9){$C$}
\rput[tl](2.9,5.0){$D$}
\psline[linewidth=1.pt,linestyle=dotted]{->}(3.02,4.74)(2.,4.)
\psline[linewidth=1.pt,linestyle=dotted]{->}(3.02,4.74)(4.,4.)
\psline[linewidth=2.pt](1.,2.)(7.,2.)
\rput[tl](1.7,1.9){$B$}\rput[tl](2.7,1.9){$A$}\rput[tl](4.7,1.9){$C$}
\rput[tl](4.9,2.9){$D$}
\psline[linewidth=1.pt,linestyle=dotted]{->}(5.07,2.64)(4.,2.)
\psline[linewidth=1.pt,linestyle=dotted]{->}(5.07,2.64)(6.,2.)
\psline[linewidth=2.pt](1.,0.)(7.,0.)
\rput[tl](2.9,0.9){$B$}\rput[tl](4.7,-0.07){$A$}\rput[tl](2.7,-0.07){$C$}
\rput[tl](5.7,-0.07){$D$}
\psline[linewidth=1.pt,linestyle=dotted]{->}(3.01,0.6)(2.,0.)
\psline[linewidth=1.pt,linestyle=dotted]{->}(3.01,0.6)(4.,0.)
\rput[tl](0.42,4.68){$(a)$}
\rput[tl](0.42,2.58){$(b)$}
\rput[tl](0.5,0.72){$(c)$}
\begin{scriptsize}
\psdots[dotstyle=o, linewidth=1.2pt,linecolor = black, fillcolor = blue]%
(3.,4.)(5.,4.)(6.,4.)
\psdots[dotstyle=o, linewidth=1.2pt,linecolor = black, fillcolor = green]%
(4.,4.)(2.,4.)
\psdots[dotstyle=o, linewidth=1.2pt,linecolor = black, fillcolor = blue]%
(2.,2.)(3.,2.)(5.06,2.)
\psdots[dotstyle=o, linewidth=1.2pt,linecolor = black, fillcolor = green]%
(6.,2.)(4.,2.)
\psdots[dotstyle=o, linewidth=1.2pt,linecolor = black, fillcolor = blue]%
(3.,0.)(5.,0.)(6.,0.)
\psdots[dotstyle=o, linewidth=1.2pt,linecolor = black, fillcolor = green]%
(4.,0.)(2.,0.)
\end{scriptsize}
\end{pspicture}
\caption[]{Possibilities of doubles
combinations, of \ $\left[ B,A,C\right] ,$ $\left[ C,A,D\right] $ and $\left[
D,A,B\right] $}
\label{fig:figure.1}
\end{figure}
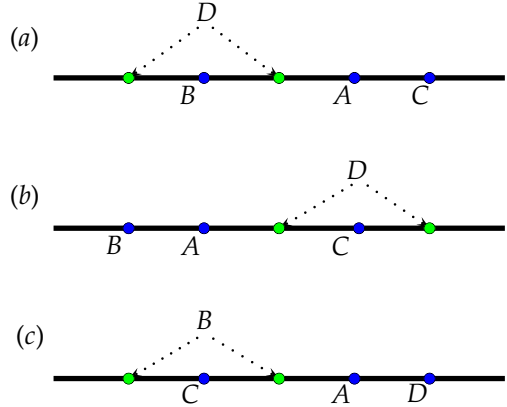

\begin{proposition}\label{Pr.1}
For all  $A,B,C,D\in \ell $ , two, of $\left[ B,A,C\right] ,$ $%
\left[ C,A,D\right] $ and $\left[ D,A,B\right] $ $\ $exclude the third.
\end{proposition}

\begin{proof}
Let's get double combinations and see how third is
excluded.

\begin{compactenum}[(a)]
\item Suppose we are true $\left[ B,A,C\right] $ and $\left[ C,A,D%
\right] ,$ \ (see Fig.~\ref{fig:figure.1}, $(a)$) then from the order axioms in Def.~\ref{def:orderAxioms}, we have:

\[
\left. 
\begin{array}{c}
\left[ B,A,C\right] \\ 
\left[ C,A,D\right]%
\end{array}%
\right\} \overset{\mathbf{Lo.1}}{\implies }\left. 
\begin{array}{c}
\left[ B,A,C\right] \\ 
\left[ D,A,C\right]%
\end{array}%
\right\} \overset{\mathbf{Lo.4}}{\implies }\left[ D,B,A\right] \vee \left[
B,D,A\right].
\]

From Axiom $\mathbf{Lo.2}$ loses the possibility that $\left[ D,A,B\right] $
is true, so $\left[ D,A,B\right] $ is false.

\item Suppose we are true $\left[ B,A,C\right] $ and $\left[ D,A,B%
\right] ,$ \  (see Fig.~\ref{fig:figure.1},  $(b)$) then from the axioms above we have:

\[
\left. 
\begin{array}{c}
\left[ B,A,C\right] \\ 
\left[ D,A,B\right]%
\end{array}%
\right\} \overset{\mathbf{Lo.1}}{\implies }\left. 
\begin{array}{c}
\left[ B,A,C\right] \\ 
\left[ B,A,D\right]%
\end{array}%
\right\} \overset{\mathbf{Lo.4}}{\implies }\left[ A,D,C\right] \vee \left[
A,C,D\right].
\]

From Axiom $\mathbf{Lo.2}$ loses the possibility that $\left[ C,A,D\right] $
be true, so $\left[ C,A,D\right] $ is false.

but each of them $\left[ A,D,C\right] \vee \left[ A,C,D\right] $ derives
that $\left[ C,A,D\right] $ is false.

\item Suppose we are true $\left[ B,A,C\right] $ and $\left[ D,A,B%
\right] ,$ \  (see Fig.~\ref{fig:figure.1}, $(c)$) then from the axioms above we have:

\[
\left. 
\begin{array}{c}
\left[ C,A,D\right] \\ 
\left[ D,A,B\right]%
\end{array}%
\right\} \overset{\mathbf{Lo.4}}{\implies }\left[ B,C,A\right] \vee \left[
C,A,B\right].
\]

From Axiom $\mathbf{Lo.2}$ loses the possibility that $\left[ B,A,C\right] $
be true, so $\left[ B,A,C\right] $ is false.

\end{compactenum}
\end{proof}


\begin{definition}
For three different points $A,B,C\in \ell $ , that, we say that points $B$
and $C$ lie on the same side of point $A$, if we have exactly one of $\left[
A,B,C\right] ,\left[ A,C,B\right] .$
\end{definition}

\begin{proposition}\label{Pr.2}
For every four different points $A,B,C,D$ in a line $\ell ,$ in Desargues
affine plane$.$ Then %
\[
\left. 
\begin{array}{c}
\lbrack A,B,C] \\ 
\lbrack A,C,D]%
\end{array}%
\right\} \Longrightarrow \left\{ 
\begin{array}{c}
\lbrack A,B,D] \\ 
\lbrack B,C,D]%
\end{array}%
\right. 
\]

\end{proposition}

\begin{proof}
By $[A,B,C],$ we have that the points $A$ and $B$ are in the same side of
point $C,$ by \ $[A,C,D],$ we have that the points $A$ and $D$ are on the
opposite side of point $C.$ Then, we have that, the point $B$ and $D$ are on
the opposite side of point $C$ (since otherwise we would have that points $A$%
, $B$ and $D$ would be on the same side of point $C$, this would exclude $%
[A,C,D]$)$,$ so we have that $[B,C,D].$

By $[A,B,C],$ we have that the points $A$ and $C$ on the opposite side of
point $B.$ Since $[B,C,D],$ we have that the points $C$ and $D$ are in the
same side of point $B.$ Then, we have that, the point $A$ and $D$ are on
opposite sides of the point $B$ (since otherwise we would have that points $%
A $ and $D$ \ would be on the same side of point $B$, and point $C$ this
will give us $[D,B,C]$ this would exclude $[B,C,D]$), so we have that $%
[A,B,D].$
\end{proof}

\begin{definition}
For three different points $A,B,C\in \ell $ , that, we say that points $B$
and $C$ lie on the same side of point $A$, if we have exactly one of $\left[
A,B,C\right] ,\left[ A,C,B\right] .$
\end{definition}

\begin{proposition}
If we have, that, for four different points $A,B,C,D\in \ell $: \ points $B$
and $C$ lie on the same side of point $A,$ and the points $C$ and $D$ lie on
the same side of point $A,$ then the points $B$ and $D$ lie on the same side
of point $A.$
\end{proposition}

\begin{proof}
Let's look at the possible cases of 'positioning' between the points,

If 
\[
\left. 
\begin{array}{c}
\left[ A,D,C\right] \\ 
\left[ B,C,A\right]%
\end{array}%
\right\} \overset{\mathbf{Lo.1}}{\Longrightarrow }\left. 
\begin{array}{c}
\left[ A,D,C\right] \\ 
\left[ A,C,B\right]%
\end{array}%
\right\} \overset{~\ref{Pr.2}}{\Longrightarrow }\left[ A,D,B\right] .
\]

If 

\[
\left. 
\begin{array}{c}
\left[ A,B,C\right] \\ 
\left[ A,C,D\right]%
\end{array}%
\right\} \overset{~\ref{Pr.2}}{\Longrightarrow }\left[ A,B,D\right].
\]

If 
\[
\left. 
\begin{array}{c}
\left[ A,D,C\right] \\ 
\left[ A,B,C\right]%
\end{array}%
\right\} \overset{~\ref{Pr.1}}{\Longrightarrow }\left[ A,B,D\right] \vee \left[ A,D,B\right] .
\]

If 
\[
\left. 
\begin{array}{c}
\left[ A,D,C\right] \\ 
\left[ B,C,A\right]%
\end{array}%
\right\} \overset{\mathbf{Lo.1}}{\Longrightarrow }\left. 
\begin{array}{c}
\left[ A,D,C\right] \\ 
\left[ A,C,B\right]%
\end{array}%
\right\} \overset{\mathbf{Lo.3}}{\Longrightarrow }\left[ A,D,B\right] .
\]

So we see that any of the possible cases stays.
The case when at least two of the four above points coincide the proof is clear.
\end{proof}

\begin{definition}
The parallel projection between the two lines in the Desargues affine plane,
will be called, a function,
\end{definition}

\[
P_{p}:\ell _{1}\rightarrow \ell _{2},\forall A,B\in \ell _{1},AB\parallel P_{p}(A)P_{p}(B).
\]

\noindent It is clear that this function is a bijection between any two lines in
Desargues affine planes.

\begin{definition}\label{Ordered.D.A.Plane}
\cite{11} A Desargues Affine plane, is said to be \textbf{ordered}, provided

\begin{enumerate}
\item All lines in this plane are \textbf{ordered lines}.

\item Parallel projection of the points of one line onto the points of
another line in this plane, either preserves or reverses the ordering.
\end{enumerate}
\end{definition}

\begin{theorem}\label{thm:orderInALine}
Each translation (dilation which is different from $id_{\mathcal{P}}$ and
has no fixed point) in a finite Desargues affine plane $\mathcal{A}_{D}=(%
\mathcal{P},\mathcal{L},\mathcal{I})$ preserves order in a line $\ell $ of
this plane.
\end{theorem}

\begin{proof}
Let $A,B,C\in \ell ,$ and $\left[ A,B,C\right] .$ Let's have, too, a
each translation $\varphi $ on this plan. \ Here we will distinguish two
cases regarding the direction of translation. 
\begin{compactenum}[{Case}.1]
\item The translation $\varphi $ has not direction according to
line $\ell ,$ in this case $\varphi \left( \ell \right) \parallel \ell ,$
and $\varphi \left( \ell \right) \neq \ell$, (see Fig.~\ref{fig:figure.2}).  Mark $\varphi \left( \ell
\right) =\ell ^{\prime }.$

\noindent From the translation properties, we have,

\[
A,B,C\in \ell \Longrightarrow \varphi \left( A\right) ,\varphi \left(
B\right) ,\varphi \left( C\right) \in \ell ^{\prime }.
\]
From the translation properties described at \cite{5}, we have the following
parallelisms,

\[
AB\parallel \varphi \left( A\right) \varphi \left( B\right) ,AC\parallel
\varphi \left( A\right) \varphi \left( C\right) ,BC\parallel \varphi \left(
B\right) \varphi \left( C\right), 
\]

\noindent and

\[
A\varphi \left( A\right) \parallel B\varphi \left( B\right) \parallel
C\varphi \left( C\right) .
\]

\noindent  So we have the following parallelograms (see \cite{6}):

\[
\left( A,B,\varphi \left( B\right) ,\varphi \left( A\right)
\right) ,\left( A,C,\varphi \left( C\right) ,\varphi \left( A\right) \right)
,\left( B,C,\varphi \left( C\right) ,\varphi \left( B\right) \right) ,
\]

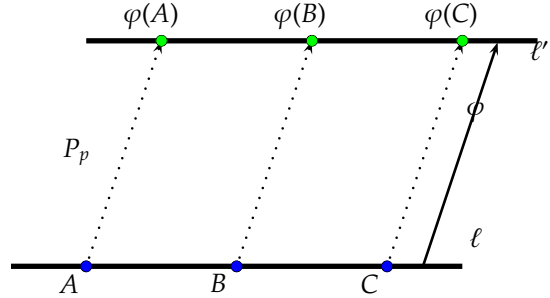
\begin{figure}[!ht]
\centering
\begin{pspicture}
(2.,0.2)(9.5,5.)
\psline[linewidth=2.pt](2.,1.)(8.,1.)
\psline[linewidth=2.pt](3.,4.)(9.,4.)
\psline[linewidth=1.pt,linestyle=dotted]{->}(3.,1.)(4.,4.)
\psline[linewidth=1.pt,linestyle=dotted]{->}(5.,1.)(6.,4.)
\psline[linewidth=1.pt,linestyle=dotted]{->}(7.,1.)(8.,4.)
\rput[tl](2.65,0.9){$A$}
\rput[tl](4.65,0.9){$B$}
\rput[tl](6.65,0.9){$C$}
\rput[tl](3.5,4.5){$\varphi(A)$}
\rput[tl](5.5,4.5){$\varphi(B)$}
\rput[tl](7.5,4.5){$\varphi(C)$}
\rput[tl](2.7,2.7){$P_p$}
\rput[tl](8.1,1.5){$\ell$}
\rput[tl](8.9,4.){$\ell'$}
\psline[linewidth=1.pt]{->}(7.47,0.98)(8.47,3.98)
\rput[tl](8.06,3.18){$\varphi$}
\begin{scriptsize}
\psdots[dotstyle=o, linewidth=1.2pt,linecolor = black, fillcolor = blue]%
(3.,1.)(5.,1.)(7.,1.)
\psdots[dotstyle=o, linewidth=1.2pt,linecolor = black, fillcolor = green]%
(4.,4.)(6.,4.)(8.,4.)
\end{scriptsize}
\end{pspicture}
\caption{The translation $\protect%
\varphi $ as a parallel projection $P_{p}:\ell \rightarrow \ell ^{\prime }.$}
\label{fig:figure.2}
\end{figure}

Clearly, we see that the translation $\varphi $ can be seen as parallel
projection $P_{p}$, from line $\ell $ to line $\ell ^{\prime }$, with
direction the line $\ell ^{A\varphi (A)}$.

Thus by definition and so $[A,B,C]$, we have $\left[ \varphi \left( A\right)
,\varphi \left( B\right) ,\varphi \left( C\right) \right] .$

\item The translation $\varphi $ has direction according to line $%
\ell ,$ in this case $\varphi \left( \ell \right) \parallel \ell ,$ and $%
\varphi \left( \ell \right) =\ell $, (see Fig.~\ref{fig:figure.3}).  By translation properties have,

\[
A,B,C\in \ell \Longrightarrow \varphi \left( A\right) ,\varphi \left(
B\right) ,\varphi \left( C\right) \in \ell .
\]

In this case, we choose a point $E\notin \ell $ of the plan. For this point, 
$\exists !\ell _{\ell }^{E}\in \mathcal{L},$ so it exists a translation $%
\varphi _{1},$ such that $\varphi _{1}\left( A\right) =E,$ and exists a
translation $\varphi _{2},$ such that $\varphi _{2}\left( E\right) =\varphi
\left( A\right) .$

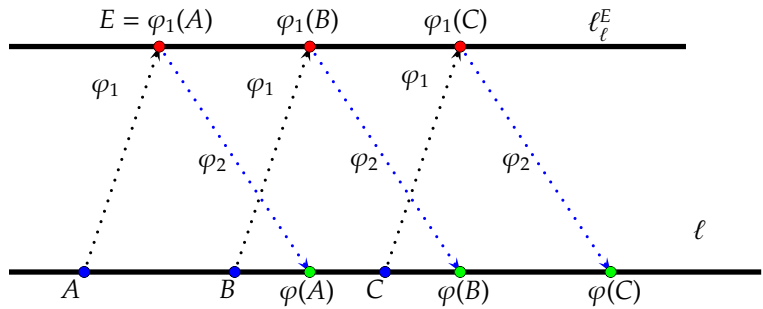
\begin{figure}[!ht]
\centering
\begin{pspicture}
(1.,0.)(11.,5.)
\psline[linewidth=2.pt](1.,1.)(11.,1.)
\psline[linewidth=2.pt](1.,4.)(10.,4.)
\psline[linewidth=1.2pt,linestyle=dotted]{->}(2.,1.)(3.,4.)
\psline[linewidth=1.2pt,linestyle=dotted]{->}(4.,1.)(5.,4.)
\psline[linewidth=1.2pt,linestyle=dotted]{->}(6.,1.)(7.,4.)
\psline[linewidth=1.2pt,linestyle=dotted,linecolor=blue]{->}(3.,4.)(5.,1.)
\psline[linewidth=1.2pt,linestyle=dotted,linecolor=blue]{->}(7.,4.)(9.,1.)
\psline[linewidth=1.2pt,linestyle=dotted,linecolor=blue]{->}(5.,4.)(7.,1.)
\rput[tl](1.7,0.9){$A$}
\rput[tl](3.8,0.9){$B$}
\rput[tl](5.75,0.9){$C$}
\rput[tl](4.6,0.9){$\varphi(A)$}
\rput[tl](6.7,0.9){$\varphi(B)$}
\rput[tl](8.7,0.9){$\varphi(C)$}
\rput[tl](2.2,4.5){$E=\varphi_{1}(A)$}
\rput[tl](4.56,4.5){$\varphi_{1}(B)$}
\rput[tl](6.52,4.5){$\varphi_{1}(C)$}
\rput[tl](8.7,4.5){$\ell^{E}_{\ell}$}
\rput[tl](10.1,1.68){$\ell$}
\rput[tl](2.1,3.56){$\varphi_{1}$}
\rput[tl](7.56,2.58){$\varphi_{2}$}
\rput[tl](4.16,3.58){$\varphi_{1}$}
\rput[tl](6.22,3.68){$\varphi_{1}$}
\rput[tl](3.52,2.56){$\varphi_{2}$}
\rput[tl](5.54,2.58){$\varphi_{2}$}
\begin{scriptsize}
\psdots[dotstyle=o, linewidth=1.2pt,linecolor = black, fillcolor = blue]%
(2.,1.)(4.,1.)(6.,1.)
\psdots[dotstyle=o, linewidth=1.2pt,linecolor = black, fillcolor = red]%
(3.,4.)(5.,4.)(7.,4.)
\psdots[dotstyle=o, linewidth=1.2pt,linecolor = black, fillcolor = green]%
(5.,1.)(9.,1.)(7.,1.)
\end{scriptsize}
\end{pspicture}
\caption{The translations $\protect%
\varphi $ as a composition of two translations $\protect\varphi _{1}$ and $%
\protect\varphi _{2}:$ $\protect\varphi =\protect\varphi _{2}\circ \protect%
\varphi _{1}.$}
\label{fig:figure.3}
\end{figure}

Well,

\[
\left( \varphi _{2}\circ \varphi _{1}\right) \left( A\right) =\varphi
\left( A\right) .
\]

Two translations $\varphi _{1}$ and $\varphi _{2}$ have different directions
from line $\ell .$ Now we repeat the first case twice, once for the
translation $\varphi _{1}$ and once for the translation $\varphi _{2}$, and
we have:

\[
\left[ A,B,C\right] \Longrightarrow \left[ \varphi _{1}\left( A\right)
,\varphi _{1}\left( B\right) ,\varphi _{1}\left( C\right) \right]
\Longrightarrow \left[ \varphi _{2}\left( \varphi _{1}\left( A\right)
\right) ,\varphi _{2}\left( \varphi _{1}\left( B\right) \right) ,\varphi
_{2}\left( \varphi _{1}\left( C\right) \right) \right] 
\]

thus,

\[
\left[ A,B,C\right] \Longrightarrow \left[ \left( \varphi_{2}\circ \varphi_{1}\right) \left( A\right) ,
\left( \varphi _{2}\circ \varphi _{1}\right)
\left( B\right) ,\left( \varphi_{2}\circ \varphi _{1}\right) \left(
C\right) \right] .
\]

Hence

\[
\left[ A,B,C\right] \Longrightarrow \left[ \varphi \left( A\right)
,\varphi \left( B\right) ,\varphi \left( C\right) \right] 
\]

\end{compactenum}

\end{proof}

\begin{theorem}\label{thm:finiteOrderedDesarguesianAffinePlane}
Every finite ordered Desarguesian affine plane over $\mathbb{R}^2$ is Pappian.
\end{theorem} 
\begin{proof}
Let $\pi$ be a finite Desarguesian affine plane over $\mathbb{R}^2$.  From Theorem~\ref{thm:Tecklenburg}, $\pi$ is Pappian.   Since every line in $\pi$ satisfies Axioms Lo.1-Lo.4 of definition~\ref{def:orderAxioms}, so every line in $\pi$ is an ordered line, also, this plane meets the conditions of the definition~\ref{Ordered.D.A.Plane}.   Consequently,  $\pi$ is a finite ordered Desarguesian affine plane. It has been observed that every Euclidean space is Pappian~\cite[\S 1.1, p. 195]{Tecklenburg1991EuclideanSpaceIsAPappianSpace}.   Hence, $\pi$ is Pappian.
\end{proof}

\section{The ordered skew-field in a line in ordered
desargues affine plane}

In \cite{3}, \cite{4}, we have shown how to construct a
skew-field over a line in Desargues affine plane. Let it be $\ell $ a line
of Desargues affine plane $\mathcal{A}_{\mathcal{D}}=\left( \mathcal{P},%
\mathcal{L},\mathcal{I}\right) .$

We mark \ $\mathbf{K}=$\ $\left( \ell ,+,\ast \right) $\ the skew-field
constructed over the line $\ell ,$ in Desargues affine plane $\mathcal{A}_{%
\mathcal{D}}.$ In previous work \cite{3}, we have shown how we can transform
a line in the Desargues affine plane into an additive Group of its points.
We have also shown \cite{4} how to construct a skew-field with a set of
points on a line in the Desargues affine plane. In addition, for a line of
in any Desargues affine plane, we construct a skew-field with the points of
this line, by appropriately defined addition and multiplication of points in
a line.

During the construction of the skew-field over a line of a Desargues affine
plane, we choose two points (each affine plan has at least two points),
which we write with $O$ and $I$ and call them zero and one, respectively.
These points play the role of unitary elements regarding the two actions
addition and multiplication, respectively.

\begin{definition}
\cite{11},\cite{10}A skew-field $\mathbf{K}$ is said to be ordered,
if, satisfies the following conditions

\begin{compactenum}[1.]
\item  $K=K_{-}\cup \left\{ 0_{K}\right\} \cup K_{+}$ and $K_{-}\cap
\left\{ 0_{K}\right\} \cap K_{+}=\varnothing ,$

\item  For all $k_{1},k_{2}\in K_{+},k_{1}+k_{2}\in K_{+},$ ($K_{+}$ is closed under addition)

\item For all $k_{1},k_{2}\in K_{+},k_{1}\ast k_{2}\in K_{+},$ (a product of positive elements is positive)
\end{compactenum}
\end{definition}

Consider a line $\ell $ in a Desargues affine plane, by defining
the affine plane so that we have at least two points in this line, which we mark $%
\mathbf{O}$ and $\mathbf{I.}$ For this line, we have shown in the previous
works (see \cite{3},\cite{4}) that we can construct a skew-field $\mathbf{K=}%
\left( \mathbf{\ell ,+,\ast }\right) $. We have shown
~\cite{3} 
the independence of choosing points $\mathbf{O}$ and $\mathbf{I}$
in a line, for the construction of a skew-field over this line.

To establish a separation of the set of points in line $\ell $, we separate, firstly
the points $\{\mathbf{O}\},$ which we mark with

\[
\mathbf{K}_{+}=\left\{ X\in \mathbf{\ell }\text{ }\mathbf{|}\text{ }[O,X,I]%
\text{ }or\text{ }[O,I,X]\right\} (=\mathbf{\ell }_{+}).
\]

\noindent and

\[
\mathbf{K}_{-}=\left\{ X\in \mathbf{\ell }\text{ }\mathbf{|}\text{ }%
[X,O,I]\right\} (=\mathbf{\ell }_{-}).
\]

So, clearly, from the definition of the ordered line in a Desargues affine
plane, we have

\[
\mathbf{K}=\mathbf{K}_{-}\cup \left\{ 0_{\mathbf{K}}\right\} \cup \mathbf{K}_{+}, 
\]

\noindent where $0_{\mathbf{K}}=O$, and

\[
\mathbf{K}_{-}\cap \left\{ 0_{\mathbf{K}}\right\} \cap \mathbf{K}_{+}=\varnothing. 
\]

\begin{lemma}\label{lemma:A+C}
For all $A,C\in \mathbf{K}_{+}\Longrightarrow A+C\in \mathbf{K}_{+}.$
\end{lemma}

\begin{proof}
Let's have the points $A,C\in \ell ,$ such that $[O,A,C]$ or 
$[O,C,A].$ Suppose we have true $\left[ O,A,C\right] ,$  from the
construction of the $A+C$ point we have that $A+C$ $\in \ell $, (see Fig.~\ref{fig:figure.4}) which was
built according to the algorithm

\[
\forall A,C\in \ell ,\left[ 
\begin{array}{l}
\mathbf{Step.1}.\exists B\notin OI \\ 
\mbox{} \\ 
\mathbf{Step.2}.\ell _{OI}^{B}\cap \ell _{OB}^{A}=D \\ 
\mbox{} \\ 
\mathbf{Step.3}.\ell _{CB}^{D}\cap OI=E%
\end{array}%
\right] \Leftrightarrow A+C=E.
\]

\begin{figure}[!ht]
\centering
\begin{pspicture}
(1.,0.)(8.,5.)
\psline[linewidth=2.pt](1.,1.)(8.,1.)
\psline[linewidth=2.pt](2.,4.)(7.,4.)
\psline[linewidth=1.2pt,linecolor=blue](2.,1.)(3.,4.)
\psline[linewidth=1.2pt,linecolor=red](3.,4.)(5.,1.)
\rput[tl](1.75,0.9){$O$}
\rput[tl](3.66,0.9){$A$}
\rput[tl](4.7,0.9){$C$}
\rput[tl](6.4,0.9){$A+C$}
\rput[tl](2.68,4.7){$B$}
\rput[tl](4.9,4.7){$D$}
\rput[tl](7.3,1.74){$\ell$}
\rput[tl](7.1,4.44){$\ell^{B}_{\ell}$}
\rput[tl](1.7,3.){$\ell^{OB}$}
\rput[tl](3.92,3.86){$\ell^{D}_{\ell^{OB}}$}
\rput[tl](6.3,3.14){$\ell^{D}_{\ell^{BC}}$}
\rput[tl](3.3,2.74){$\ell^{BC}$}
\psline[linewidth=1.2pt,linecolor=blue](4.,1.)(5.,4.)
\psline[linewidth=1.2pt,linecolor=red](5.,4.)(7.,1.)
\begin{scriptsize}
\psdots[dotstyle=o, linewidth=1.2pt,linecolor = black, fillcolor = blue]%
(2.,1.)(4.,1.)(5.,1.)(4.,1.)
\psdots[dotstyle=o, linewidth=1.2pt,linecolor = black, fillcolor = red]%
(7.,1.)
\psdots[dotstyle=o, linewidth=1.2pt,linecolor = black, fillcolor = green]%
(3.,4.)(5.,4.)
\end{scriptsize}
\end{pspicture}
\caption{The addition of points in a
line of Desargues affine plane, as a composition of tow parallel projections.}
\label{fig:figure.4}
\end{figure}
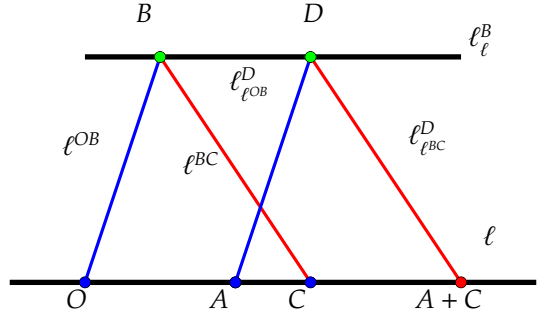

\noindent From the construction of the $A+C$ point, we have that:

\[
\ell ^{OB}\parallel \ell ^{AD};\ell ^{BC}\parallel \ell ^{D(A+C)},\ell^{BD}\parallel OI.
\]

\noindent From these parallelisms, there is the translation $\varphi _{1}$ with direction,
according to line $\ell ,$ such that, $\varphi_{1}(O)=A.$

\[
\left[ O,A,C\right] =\left[ O,\varphi_{1}(O),C\right] .
\]


\noindent Since, by with Theorem~\ref{thm:orderInALine}, the translations preseve the line-order, we have

\begin{align*}
\left[ O,A,C\right] =\left[ O,\varphi_{1}(O),C\right] & \Longrightarrow %
\left[ \varphi_{1}\left( O\right) ,\varphi_{1}\left( \varphi_{1}(O)\right) ,\varphi_{1}(C)\right]\\
           & \Longrightarrow
\left[ \varphi_{1}\left( O\right) ,\varphi_{1}\left( A\right) ,\varphi_{1}(C)\right]
=[A,\varphi_{1}(A),\varphi_{1}(C)].
\end{align*}

\noindent Hence, we have defined a translation (see \cite{1},\cite{2}), such that we have
true

\[
[O,\varphi_{1}(O)=A,\varphi _{1}(A)]=[O,A,\varphi_{1}(A)]
\]

\noindent From this translation, we also have,

\[
\varphi_{1}(B)=D\ \mbox{and}\ \varphi_{1}(C)=A+C,
\]


\noindent  By by with Theorem~\ref{thm:orderInALine} we have

\[
\left[ O,A,C\right] \Longrightarrow \left[ \varphi_{1}(O),\varphi_{1}\left( A\right) ,
\varphi_{1}(C)\right] =\left[ A,\varphi_{1}\left(A\right) ,A+C\right].
\]

For four points $O,A,\varphi_{1}\left( A\right) ,A+C$ $\in \ell ,$ we have
true

\[
\left. 
\begin{array}{c}
\lbrack O,A,\varphi_{1}(A)] \\ 
\left[ A,\varphi_{1}\left( A\right) ,A+C\right]%
\end{array}%
\right\} \overset{\mathbf{Lo.3}}{\implies }\left[ O,A,A+C\right] \wedge %
\left[ O,\varphi_{1}(A),A+C\right] .
\]

By, $\left[ O,A,A+C\right] ,$ we have that

\[
A+C\in K_{+}(=\ell _{+}).
\]

\end{proof}

\begin{corollary}
For all $A,C\in \mathbf{K}%
_{-}\Longrightarrow A+C\in \mathbf{K}_{-}$.
\end{corollary}

\begin{corollary}
For all $A\in \mathbf{K}_{+}\mathbf{\Longrightarrow }-A\in \mathbf{K}_{-}$.
\end{corollary}

\begin{proof}
If $-A\in \mathbf{K}_{+}\Longrightarrow A+(-A)=O(=0_{K})\in \mathbf{K}%
_{+},$ which is a contradiction.
\end{proof}

\begin{lemma}\label{lemma:A*C}
For all $A,C\in \mathbf{K}_{+}\Longrightarrow A\ast C\in \mathbf{K}_{+}.$
\end{lemma}

\begin{proof}
Let's have the points $A,C\in \ell ,$ such that $[O,A,C]$ or 
$[O,C,A].$ Suppose we have true $\left[ O,A,C\right] ,$and suppose also that
we have $[O,I,A]$ \ from the construction of the $A\ast C$ point we have
that $A\ast C$ $\in \ell $, which was built according to the algorithm

\[
\forall A,C\in \ell ,\left[ 
\begin{array}{l}
\mathbf{Step.1}.\exists B\notin OI \\ 
\mbox{} \\
\mathbf{Step.2}.\ell _{IB}^{A}\cap OB=E \\ 
\mbox{} \\ 
\mathbf{Step.3}.\ell _{BC}^{E}\cap OI=F%
\end{array}%
\right] \Leftrightarrow A\ast C=F.
\]

By, construction of the point $A\ast C$, we have parallelisms

\[
IB||AE,BC||E(A\ast C).
\]

We take a parallel projection $P_{p}$ with the direction of the line $\ell^{IB}$, (see Fig.~\ref{fig:figure.5}) such that

\[
P_{p}:\ell ^{OI}\rightarrow \ell ^{OB},
P_{p}(O)=O;P_{p}(I)=B 
\]

and $P_{p}(A)=E$, since, parallel projection, preserve the order, we have

\begin{figure}[!ht]
\centering
\begin{pspicture}
(0.5,0.)(10.,6.5)
\psline[linewidth=2.pt](1.,1.)(10.,1.)
\psline[linewidth=2.pt](1.,1.)(6.,6.)
\psline[linewidth=1.2pt](3.,1.)(3.,3.)
\psline[linewidth=1.2pt](4.,1.)(4.,4.)
\psline[linewidth=1.2pt,linecolor=blue](3.,3.)(6.,1.)
\psline[linewidth=1.2pt,linecolor=blue](4.,4.)(8.42,1.)
\rput[tl](2.5,0.9){$I$}
\rput[tl](4.,0.9){$A$}
\rput[tl](5.7,0.9){$C$}
\rput[tl](7.9,0.9){$A*C$}
\rput[tl](2.35,3.1){$B$}
\rput[tl](3.2,4.1){$E$}
\psline[linewidth=1.pt,linestyle=dotted,linecolor=red]{->}(1.76,3.86)(4.564615384615385,1.9569230769230768)
\psline[linewidth=1.pt,linestyle=dotted,linecolor=red]{->}(3.,0.)(3.,2.)
\psline[linewidth=1.pt,linestyle=dotted,linecolor=red]{->}(4.,0.)(4.,2.)
\rput[tl](0.82,0.8){$O$}
\psline[linewidth=1.2pt,linestyle=dotted,linecolor=red]{->}(2.82,4.8)(5.891143662129771,2.716418328871196)
\rput[tl](3.2,0.7){$P_p$}
\rput[tl](2.5,4.5){$\widetilde{P_{p}} $}
\rput[tl](4.7,4.7){$\ell^{OB}$}
\rput[tl](8.9,1.7){$\ell^{OI}$}
\rput[tl](2.3,1.9){$\ell^{IB}$}
\rput[tl](5.2,2.14){$\ell^{BC}$}
\rput[tl](4.12,3.){$\ell^{A}_{\ell^{IB}}$}
\rput[tl](3.56,5.06){$\ell^{E}_{\ell^{BC}}$}
\begin{scriptsize}
\psdots[dotstyle=o, linewidth=1.2pt,linecolor = black, fillcolor = blue]%
(1.,1.)(3.,1.)(4.,1.)(6.,1.)
\psdots[dotstyle=o, linewidth=1.2pt,linecolor = black, fillcolor = red]%
(8.42,1.)
\psdots[dotstyle=o, linewidth=1.2pt,linecolor = black, fillcolor = green]%
(3.,3.)(4.,4.)
\end{scriptsize}
\end{pspicture}
\caption{The multiplication of points in
a line of Desargues affine plane, as a composition of tow parallel
projections.}
\label{fig:figure.5}
\end{figure}

\[
\left. 
\begin{array}{c}
\lbrack O,I,A] \\ 
IB\parallel AE%
\end{array}%
\right\} \Longrightarrow \left[ P_{p}(O),P_{p}(I),P_{p}(A)\right] =[O,B,E]
\]

But by the multiplication algorithm, during the construction of the $A\ast C$
point, we have $BC||E(A\ast C)$.

We take a parallel projection $\widetilde{P_{p}}$ with the direction of the
line $\ell ^{BC}$, (see Fig.~\ref{fig:figure.5}) such that

\[
\widetilde{P_{p}}:\ell ^{OB}\rightarrow \ell ^{OI},
\widetilde{P_{p}}(O)=O;\widetilde{P_{p}}(B)=C \text{ }and\text{ } \widetilde{P_{p}}(E)=A\ast C.
\]

since, parallel projection, preserve the order, we have

\[
\left. 
\begin{array}{c}
[O,B,E] \\ 
BC||E(A\ast C)%
\end{array}%
\right\} \Longrightarrow \left[ \widetilde{P_{p}}(O),\widetilde{P_{p}}(B),%
\widetilde{P_{p}}(E)\right] =[O,C,A\ast C]
\]

By, $\left[ O,C,A\ast C\right] ,$ we have that

\[
A\ast C\in K_{+}(=\ell _{+}).
\]

If we have, $[O,A,C]$ and $[A,I,C]\Longrightarrow \lbrack O,A,I],$ (see Fig.~\ref{fig:figure.6}). By,
construction of the point $A\ast C$, we have parallelisms

\[
IB||AE,BC||E(A\ast C).
\]

We take a parallel projection $P_{p}$ with the direction of the line $\ell
^{IB}$, such that

\[
P_{p}:\ell ^{OI}\rightarrow \ell ^{OB},
P_{p}(O)=O; P_{p}(A)=E\ \mbox{and}\ P_{p}(I)=B.
\]

since, parallel projection, preserve the order, we have 
\[
[O,A,I]\Longrightarrow \lbrack P_{p}(O),P_{p}(A),P_{p}(I)]\Longrightarrow
\lbrack O,E,B].
\]

\begin{figure}[!ht]
\centering
\begin{pspicture}
(0.5,0.)(10.,6.5)
\psline[linewidth=2.pt](1.,1.)(10.,1.)
\psline[linewidth=2.pt](1.,1.)(6.,6.)
\psline[linewidth=1.2pt](3.,1.)(3.,3.)
\psline[linewidth=1.2pt](4.,1.)(4.,4.)
\psline[linewidth=1.2pt,linecolor=blue](3.,3.)(6.,1.)
\psline[linewidth=1.2pt,linecolor=blue](4.,4.)(8.42,1.)
\rput[tl](4.24,0.9){$I$}
\rput[tl](2.46,0.9){$A$}
\rput[tl](8.2,0.9){$C$}
\rput[tl](5.4,0.9){$A*C$}
\rput[tl](3.7,4.7){$B$}
\rput[tl](2.22,3.3){$E$}
\psline[linewidth=1.2pt,linestyle=dotted,linecolor=red]{->}(1.76,3.86)(4.564615384615385,1.9569230769230768)
\psline[linewidth=1.2pt,linestyle=dotted,linecolor=red]{->}(3.,0.)(3.,2.)
\psline[linewidth=1.2pt,linestyle=dotted,linecolor=red]{->}(4.,0.)(4.,2.)
\rput[tl](0.7,0.7){$O$}
\psline[linewidth=1.2pt,linestyle=dotted,linecolor=red]{->}(2.82,4.8)(5.891143662129771,2.716418328871196)
\rput[tl](3.2,0.7){$P_p$}
\rput[tl](2.6,4.3){$\widetilde{P_{p}}$}
\rput[tl](4.8,4.8){$\ell^{OB}$}
\rput[tl](8.98,1.5){$\ell^{OI}$}
\rput[tl](3.25,3.4){$\ell^{IB}$}
\rput[tl](6.4,3.){$\ell^{BC}$}
\rput[tl](2.1,2.2){$\ell^{A}_{\ell^{IB}}$}
\rput[tl](4.76,2.7){$\ell^{E}_{\ell^{BC}}$}
\begin{scriptsize}
\psdots[dotstyle=o, linewidth=1.2pt,linecolor = black, fillcolor = blue]%
(1.,1.)(3.,1.)(4.,1.)(6.,1.)
\psdots[dotstyle=o, linewidth=1.2pt,linecolor = black, fillcolor = green]%
(3.,3.)(4.,4.)
\psdots[dotstyle=o, linewidth=1.2pt,linecolor = black, fillcolor = red]%
(8.42,1.)
\end{scriptsize}
\end{pspicture}
\caption{The case where, we have, $%
[O,A,C]$ and $[A,I,C].$}
\label{fig:figure.6}
\end{figure}

But by the multiplication algorithm, during the construction of the $A\ast C$
point, we have $BC||E(A\ast C)$.We take a parallel projection $\widetilde{P_{p}}$ 
with the direction of the line $\ell ^{BC}$, such that

\[
\widetilde{P_{p}}:\ell ^{OB}\rightarrow \ell ^{OI},
\widetilde{P_{p}}(O)=O;\widetilde{P_{p}}(E)=A\ast C\ \mbox{and}\ \widetilde{P_{p}}%
(B)=C.
\]

since, parallel projection, preserve the order, we have

\[
[O,E,B]\Longrightarrow \left[ \widetilde{P_{p}}(O),\widetilde{P_{p}}(E),%
\widetilde{P_{p}}(B)\right] =[O,A\ast C,C]\Longrightarrow A\ast C\in
K_{+}(=\ell _{+}).
\]

If we have, $[O,A,C]$ and $[A,C,I]\Longrightarrow [O,A,I]\wedge
[O,C,I],$ (see Fig.~\ref{fig:figure.7}). By, construction of the point $A\ast C$, we have
parallelisms

\[
IB||AE,BC||E(A\ast C).
\]

We take a parallel projection $P_{p}$ with the direction of the line $\ell
^{IB}$, such that

\[
P_{p}:\ell ^{OI}\rightarrow \ell ^{OB},
P_{p}(O)=O; P_{p}(A)=E\ \mbox{and}\ P_{p}(I)=B.
\]

since, parallel projection, preserve the order, we have 
\[
[O,A,I]\Longrightarrow \lbrack P_{p}(O),P_{p}(A),P_{p}(I)]\Longrightarrow
\lbrack O,E,B].
\]

\begin{figure}[!ht]
\centering
\begin{pspicture}
(0.5,0.)(9.,5.)
\psline[linewidth=2.pt](1.,1.)(10.,1.)
\psline[linewidth=2.pt](1.,1.)(8.06,4.56)
\psline[linewidth=1.2pt](2.8,1.)(4.0189966281279395,2.522326911633918)
\psline[linewidth=1.2pt](4.76,1.)(6.989750020794277,4.0203272059529205)
\psline[linewidth=1.2pt,linecolor=blue](4.0189966281279395,2.522326911633918)(4.,1.)
\psline[linewidth=1.2pt,linecolor=blue](6.989750020794277,4.0203272059529205)(7.,1.)
\rput[tl](7.1,0.9){$I$}
\rput[tl](3.42,0.9){$A$}
\rput[tl](4.54,0.9){$C$}
\rput[tl](2.2,0.9){$A*C$}
\rput[tl](6.38,4.5){$B$}
\rput[tl](3.44,3.12){$E$}
\psline[linewidth=1.2pt,linestyle=dotted,linecolor=red]{->}(4.,0.)(4.009089155919676,1.7283748761817668)
\psline[linewidth=1.2pt,linestyle=dotted,linecolor=red]{->}(5.04,3.78)(3.4094983140639696,1.7611634558169587)
\rput[tl](0.7,0.9){$O$}
\psline[linewidth=1.2pt,linestyle=dotted,linecolor=red]{->}(7.,0.)(6.9941355826069165,2.7280483251619736)
\rput[tl](5.38,0.6){$P_p$}
\rput[tl](5.7,4.82){$\widetilde{P_{p}} $}
\rput[tl](7.66,4.24){$\ell^{OB}$}
\rput[tl](8.1,1.72){$\ell^{OI}$}
\rput[tl](7.3,2.66){$\ell^{IB}$}
\rput[tl](5.94,2.76){$\ell^{BC}$}
\rput[tl](4.16,2.6){$\ell^{A}_{\ell^{IB}}$}
\rput[tl](4.06,4.18){$\ell^{E}_{\ell^{BC}}$}
\psline[linewidth=1.2pt,linestyle=dotted,linecolor=red]{->}(7.66,4.94)(6.319248334858248,3.15149066176988)
\begin{scriptsize}
\psdots[dotstyle=o, linewidth=1.2pt,linecolor = black, fillcolor = blue]%
(1.,1.)(2.8,1.)(4.76,1.)(4.,1.)
\psdots[dotstyle=o, linewidth=1.2pt,linecolor = black, fillcolor = green]%
(4.0189966281279395,2.522326911633918)(6.989750020794277,4.0203272059529205)
\psdots[dotstyle=o, linewidth=1.2pt,linecolor = black, fillcolor = blue]%
(7.,1.)
\end{scriptsize}
\end{pspicture}
\caption{The case wherewe have, $[O,A,C]$
and $[A,C,I].$}
\label{fig:figure.7}
\end{figure}

But by the multiplication algorithm, during the construction of the $A\ast C$
point, we have $BC||E(A\ast C)$.We take a parallel projection $\widetilde{%
P_{p}}$ with the direction of the line $\ell ^{BC}$, such that

\[
\widetilde{P_{p}}:\ell ^{OB}\rightarrow \ell ^{OI},
\widetilde{P_{p}}(O)=O;\widetilde{P_{p}}(E)=A\ast C;  %
\widetilde{P_{p}}(B)=C,
\]

and $\widetilde{P_{p}}$, since parallel projection preserves the order, we have

\[
[O,E,B]\Longrightarrow \left[ \widetilde{P_{p}}(O),\widetilde{P_{p}}(E),%
\widetilde{P_{p}}(B)\right] =[O,A\ast C,C]\Longrightarrow A\ast C\in
K_{+}(=\ell _{+}).
\]

\end{proof}

We have indirectly from Lemma~\ref{lemma:A*C} the case when $[O,C,A]$,  since $A\ast C\neq C\ast A.$

From Lemmas~\ref{lemma:A+C} and ~\ref{lemma:A*C}, we have the following result.

\begin{theorem}\label{thm:OrderedSkewFieldDesargues}
\bigskip Every skew-field that is constructed over an ordered-line of a Desargues affine plane is an ordered skew-field.
\end{theorem}

If we consider the special case where ordered-line of a Desargues affine plane in $\mathbb{R}^2$, we obtain the following result.

\begin{corollary}\label{cor:finiteOrderedSkewField}
\bigskip Every finite skew-field that is constructed over an ordered-line of a Desargues affine plane in $\mathbb{R}^2$ is a finite ordered skew-field.
\end{corollary}

Putting together the result from Corollary~\ref{cor:finiteOrderedSkewField} and Theorem~\ref{thm:finiteOrderedDesarguesianAffinePlane}, we obtain the following result for an ordered lines o a finite Desargues affine plane in $\mathbb{R}^2$.

\begin{theorem}\label{cor:finiteOrderedSkewFieldR2}
A finite skew field constructed over an ordered-line on a finite Desargues affine plane in $\mathbb{R}^2$ is Pappian.
\end{theorem}
\begin{proof}
Let $\ell$ be an ordered-line in a finite Desargues affine plane in $\mathbb{R}^2$.   By definition, $\ell$ is a Desarguesian affine 1-plane.   From Corollary~\ref{cor:finiteOrderedSkewField}, $\ell$ is a finite ordered skew-field.
Hence, from Theorem~\ref{thm:finiteOrderedDesarguesianAffinePlane}, $\ell$ is Pappian.
\end{proof}

\bibliographystyle{amsplain}
\bibliography{NSrefs}

\end{document}